\documentclass[reqno]{amsart}

\usepackage{amsfonts}
\usepackage{amssymb}
\usepackage{amsmath}
\usepackage{amsthm}
\usepackage{bbm}
\usepackage{graphics}
\usepackage{mathrsfs}

\newcommand{\zz}{\mathbb{Z}}

\newcommand{\cc}{\mathbb{C}}

\newcommand{\qq}{\mathbb{Q}}

\newcommand{\gf}{\mathbb{F}}

\newcommand{\legendre}[2]{\left(\frac{#1}{#2}\right)}
\newcommand{\onetoone}{\hookrightarrow}

\newcommand{\dnd}{\nmid}
\newcommand{\isom}{\cong}

\newcommand{\Aut}{\mbox{Aut}}
\newcommand{\GL}{\mbox{GL}}
\newcommand{\ord}{\mathrm{ord}}
\newcommand{\abcd}{\left(\begin{array}{cc} 
a & b \\ c & d
\end{array}\right)}

\newtheorem{prop}{Proposition}
\newtheorem{thm}{Theorem}
\newtheorem{cor}{Corollary}
\newtheorem{defn}{Definition}

\usepackage{fullpage}
\begin{document}
\title{Elliptic Curves, Modular Forms, and Sums of Hurwitz Class Numbers\\
(appeared in \textit{Journal of Number Theory})}

\bibliographystyle{plain}
\author{Brittany Brown}
\address{
Brittany Brown\\
Department of Mathematics and Actuarial Science\\
Butler University\\
4600 Sunset Ave.\\
Indianapolis, IN 46208
}

\author{Neil J. Calkin}
\address{
Neil J. Calkin\\
Department of Mathematical Sciences\\
Clemson University\\
Box 340975 Clemson, SC 29634-0975
}
\email{calkin@clemson.edu}

\author{Timothy B. Flowers}
\address{
Timothy B. Flowers\\
Department of Mathematical Sciences\\
Clemson University\\
Box 340975 Clemson, SC 29634-0975
}
\email{tflower@clemson.edu}

\author{Kevin James}
\address{
Kevin James\\
Department of Mathematical Sciences\\
Clemson University\\
Box 340975 Clemson, SC 29634-0975
}
\email{kevja@clemson.edu}

\author{Ethan Smith}
\address{
Ethan Smith\\
Department of Mathematical Sciences\\
Clemson University\\
Box 340975 Clemson, SC 29634-0975
}
\email{ethans@math.clemson.edu}

\author{Amy Stout}
\address{
Amy Stout\\
Department of Mathematics\\
LeConte College\\
1523 Greene Street\\
University of South Carolina\\
Columbia, SC 29298
}
\thanks{This work was partially funded by the NSF grant DMS: 0244001.}

\begin{abstract}
Let $H(N)$ denote the Hurwitz class number.
It is known that if $p$ is a prime, then
\begin{equation*}
\sum_{|r|<2\sqrt p}H(4p-r^2) = 2p.
\end{equation*}
In this paper, we investigate the behavior of this sum with 
the additional condition $r\equiv c\pmod m$.  Three different 
methods will be explored for determining the values of such sums.
First, we will count isomorphism classes of elliptic curves over finite fields.
Second, we will express the sums as coefficients of modular forms.
Third, we will manipulate the Eichler-Selberg trace formula for Hecke operators
to obtain Hurwitz class number relations. 
The cases $m=2,3$ and $4$ are treated in full.  Partial results, as well as 
several conjectures, are 
given for $m=5$ and $7$.
\end{abstract}

\maketitle

\section{Introduction and Statement of Theorems}
We begin by recalling the definition of the Hurwitz class number.
\begin{defn}
For an integer $N\geq 0$, the Hurwitz class number
$H(N)$ is defined as follows.
$H(0)=-1/12$.
If $N\equiv 1$ or $2\pmod 4$, then $H(N)=0$.
Otherwise, $H(N)$ is the number of classes of not necessarily primitive  
positive definite quadratic forms of discriminant $-N$, except that those 
classes which have a representative which is a multiple of the form $x^2+y^2$ 
should be 
counted with 
weight $1/2$ and those which have a representative which is a multiple of the 
form 
$x^2+xy+y^2$ 
should be counted with weight $1/3$.
\end{defn}
Several nice identities are known for sums of Hurwitz class numbers.
For example, it is known that if $p$ is a prime, then
\begin{equation}\label{fullsum}
\sum_{|r|<2\sqrt p}H(4p-r^2) = 2p,
\end{equation}
where the sum is over integers $r$ (both positive, negative, and zero).
See for example~\cite[p. 322]{cox:1989} or~\cite[p. 154]{eic:1956}.

In this paper, we investigate the behavior of this sum with 
additional condition $r\equiv c\pmod m$.
In particular, 
if we split the sum according to the parity of $r$, then 
we have
\begin{thm}\label{mod2split}
 If $p$ is an odd prime, then
\begin{equation*}
\sum_{\substack{|r|<2\sqrt p,\\r\equiv c\pmod 2}}H(4p-r^2)
=\begin{cases}
 \frac{4p-2}{3},&\mbox{ if } c=0,\\ 
 \frac{2p+2}{3},&\mbox{ if } c=1.
 \end{cases}
\end{equation*}
\end{thm}

Once we have the above result, we can use the ideas in its proof to 
quickly prove the next.
\begin{thm}\label{mod4split}
If $p$ is an odd prime,
\begin{equation*}
\sum_{\substack{|r|<2\sqrt p,\\r\equiv c\pmod 4}}H(4p-r^2)
=\begin{cases}
\frac{p+1}{3}, & c\equiv\pm 1\pmod 4,\\
\frac{5p-7}{6}, & c\equiv p+1\pmod 4,\\
\frac{p+1}{2}, & c\equiv p-1\pmod 4.
\end{cases}
\end{equation*}
\end{thm}

We will also fully characterize the
case $m=3$ by proving the following formulae.
\begin{thm}\label{mod3split}
If $p$ is prime, 
then
\begin{eqnarray*}
\sum_{\substack{|r|<2\sqrt p,\\r\equiv c\pmod 3}}H(4p-r^2)
=\begin{cases}
\frac{p+1}{2}, &\mbox{ if } c\equiv 0\pmod 3,\ p\equiv 1\pmod 3,\\
p-1, &\mbox{ if } c\equiv 0\pmod 3,\ p\equiv 2\pmod 3,\\
\frac{3p-1}{4}, &\mbox{ if } c\equiv\pm 1\pmod 3,\ p\equiv 1\pmod 3,\\
\frac{p+1}{2}, &\mbox{ if } c\equiv\pm 1\pmod 3,\ p\equiv 2\pmod 3.
\end{cases}
\end{eqnarray*}
\end{thm}

We also have a partial characterization for the sum split 
according to the value of $r$ modulo 5.
\begin{thm}\label{partialmod5}
If $p$ is prime, then
\begin{equation*}
\sum_{\substack{|r|<2\sqrt p,\\ r\equiv c\pmod 5}}H(4p-r^2)=
\begin{cases}
\frac{p-1}{2},&\mbox{ if } c\equiv\pm(p+1)\pmod 5,\
p\equiv\pm 2\pmod 5,\\
\frac{p-3}{2},&\mbox{ if } c\equiv 0\pmod 5,\ p\equiv 4\pmod 5.
\end{cases}
\end{equation*}
\end{thm}

All of the above theorems may be proven by exploiting the relationship 
between Hurwitz class numbers and elliptic curves over finite fields.  In 
Section~\ref{elliptic_method}, we will state this relationship and show 
how it is used to 
prove Theorem~\ref{mod2split}.  We will then briefly sketch how to use 
the same method for the proof of Theorem~\ref{mod4split} as well as 
several cases of Theorem~\ref{partialmod7} below.

In Section~\ref{mf_method}, we will use a result about the modularity of 
certain ``partial" generating functions for the Hurwitz class number to 
prove Theorem~\ref{mod3split}.  The interesting thing about this method is 
that it leads to a far more general result than what is obtainable by the 
method of Section~\ref{elliptic_method}.  Out of this result, it is possible 
to extract a version of Theorem~\ref{mod3split} for $p$ not necessarily 
prime as well as the following.
\begin{thm}\label{anothresult}
If $(n,6)=1$ and there exists a prime $p\equiv 2\pmod 3$ such that 
$\mbox{ord}_p(n)\equiv 1\pmod 2$, then
\begin{equation*}
\sum_{\substack{|r|<\sqrt n,\\r\equiv c_n\pmod 3}}H(n-r^2)
=\frac{\sigma(n)}{12},
\end{equation*}
where we take $c_n=0$ if $n\equiv 1\pmod 3$, and $c_n=1$ or $2$ if 
$n\equiv 2\pmod 3$.
\end{thm}

A third method of proof will be discussed in Section~\ref{trace_method}, which 
uses the Eichler-Selberg trace formula.  
This method will allow us to prove the cases of the following result that  
remain unproven at the end of Section~\ref{elliptic_method}.
\begin{thm}\label{partialmod7}
If $p$ is prime, then
\begin{equation*}
\sum_{\substack{|r|<2\sqrt p,\\ r\equiv c\pmod 7}}H(4p-r^2)=
\begin{cases}
\frac{p+1}{3}, & c\equiv 0\pmod 7,\ p\equiv 3,5\pmod 7,\\
\frac{p-5}{3}, & c\equiv 0\pmod 7,\ p\equiv 6\pmod 7,\\
\frac{p-2}{3}, & c\equiv \pm(p+1)\pmod 7,\ p\equiv 2,3,4,5\pmod 7,\\
\frac{p+1}{3}, & c\equiv\pm 2\pmod 7,\ p\equiv 6\pmod 7.
\end{cases}
\end{equation*}
\end{thm}

Finally, in Section~\ref{conj}, we list several conjectures, which are 
strongly supported by computational evidence.  We also give a few partial 
results 
and discuss strategies for future work.

\section{Elliptic Curves and Hurwitz Class Numbers}\label{elliptic_method}

The proofs we give in this section are 
combinatorial in nature and 
depend on the following, which is due to Deuring.
\begin{thm}[See~\cite{deu:1941} or~\cite{len:1987}]
If r is an integer such that $|r|<2\sqrt p$, then the number of isomorphism 
classes of elliptic curves over $\gf_p$ with exactly $p+1-r$ points is equal 
to the number of equivalence classes of binary quadratic forms with
discriminant $r^2-4p$.
\end{thm}
\begin{cor}\label{thecor}
For $|r|<2\sqrt p$, the number of isomorphism classes of elliptic curves over 
$\gf_p$ with exactly $p+1-r$ points is given by
$H(4p-r^2)+c_{r,p}$,
where
\begin{equation}\label{crp}
c_{r,p}=\begin{cases}
1/2, &\mbox{if }\ r^2-4p = -4\alpha^2\mbox{ for some }\alpha\in\zz,\\
2/3, &\mbox{if }\ r^2-4p = -3\alpha^2\mbox{ for some }\alpha\in\zz,\\
0, & \mbox{otherwise}.
        \end{cases}
\end{equation}
\end{cor}

Thus, the number of isomorphism classes of elliptic curves $E/\gf_p$ such 
that $m|\#E(\gf_p)$ is equal to
\begin{equation}\label{numb_classes}
\sum_{\substack{|r|<2\sqrt p\\r\equiv p+1\pmod m}}
\left(H(4p-r^2)+c_{r,p}\right).
\end{equation}
This is the main fact that we will exploit in this section.
Another useful fact that we will exploit throughout the paper is the 
symmetry of our sums.  In particular,
\begin{equation}\label{symmetry}
\sum_{\substack{|r|<2\sqrt p\\r\equiv c\pmod m}}H(4p-r^2)
=\sum_{\substack{|r|<2\sqrt p\\r\equiv -c\pmod m}}H(4p-r^2).
\end{equation}

\vspace{12pt}
\noindent\textit{Proof of Theorem~\ref{mod2split}}.

For $p=3$, the identities may be checked by direct calculation.  For the 
remainder of 
the proof, we will assume $p$ is prime and strictly greater than 3.  

Let $N_{2,p}$ denote the number of isomorphism classes of elliptic curves over 
$\gf_p$ possessing $2$-torsion,
and recall that $E$ has $2$-torsion if and only if 
$2|(p+1-r)$.  Thus, 
\begin{equation}\label{Nmp}
N_{2,p}=
\sum_{\substack{|r|<2\sqrt p\\r\equiv p+1\pmod 2}}
\left(H(4p-r^2)+c_{r,p}\right).
\end{equation}

We will proceed by 
computing $N_{2,p}$, the number of isomorphism classes of elliptic 
curves possessing 2-torsion over $\gf_p$.  
Then, we will compute the correction term,
$\sum c_{r,p}$.  In light of (\ref{fullsum}) and (\ref{Nmp}), 
Theorem 1 will follow.

We first recall the relevant background concerning elliptic 
curves with 2-torsion over $\gf_p$.  The reader is referred 
to~\cite{knapp:1992} or~\cite{sil:1986} for more details.
If $E$ is an elliptic curve with 2-torsion 
then, we can move a point of order 2 to the origin in order to 
obtain a model for $E$ of the form
\begin{equation}\label{2torcurve}
E_{b,c}: y^2 = x^3+bx^2+cx.
\end{equation}
The discriminant of such a curve is given by
\begin{equation}
\Delta = 16c^2(b^2-4c).
\end{equation}
We will omit from consideration those pairs $(b,c)$ for which the resulting 
curve has zero discriminant since these curves are singular.

Following~\cite[pp. 46-48]{sil:1986}, we take
$c_4=16(b^2-3c)$ and
$c_6=32b(9c-2b^2)$. Then since char$(\gf_p)\neq 2,3$, 
$E_{b,c}$ is isomorphic to the curve
\begin{equation}\label{c4c6COV}
E': y^2 = x^3-27c_4x-54c_6.
\end{equation}
The curves in this form that are isomorphic to (\ref{c4c6COV}) are 
\begin{eqnarray}
y^2=x^3-27u^4c_4x-54u^6c_6,\ u\neq 0.
\end{eqnarray}
Thus, given any elliptic curve, the number of $(A,B)\in\gf_p^2$ for which the
given curve is isomorphic to $E: y^2=x^3+Ax+B$ is
\begin{equation*}
 \begin{cases}
  \frac{p-1}{6}, & \mbox{ if } A=0\mbox{ and } p\equiv 1\pmod 3,\\
  \frac{p-1}{4}, &\mbox{ if } B=0\mbox{ and } p\equiv 1\pmod 4,\\
  \frac{p-1}{2}, &\mbox{ otherwise}.
 \end{cases}
\end{equation*}
We are interested in how many curves $E_{b,c}$ give the same $c_4$ and $c_6$
coefficients.  
Given an elliptic curve $E: y^2=x^3+Ax+B$
with 2-torsion over $\gf_p$, each choice of an order 2 point to
be moved to the origin yields a different model $E_{b,c}$.  
Thus, the number of $E_{b,c}$
which have the same $c_4$ and $c_6$ coefficients is equal to the number of 
order 2 points possessed by the curves.  This is either 1 or 3 depending on 
whether the curves have full or cyclic 2-torsion.  Thus, the number of 
$(b,c)$ for which $E_{b,c}$ is isomorphic to a given curve is
\begin{equation}\label{2torclasscount}
\begin{cases}
\frac{p-1}{6}, & c_4=0,\ p\equiv 1\pmod 3\mbox{ and 2-torsion is cyclic},\\
\frac{p-1}{4}, & c_6=0,\ p\equiv 1\pmod 4\mbox{ and 2-torsion is cyclic},\\
\frac{p-1}{2}, &\mbox{otherwise with cyclic 2-torsion},\\
\frac{p-1}{2}, & c_4=0,\ p\equiv 1\pmod 3\mbox{ and 2-torsion is full},\\
\frac{3(p-1)}{4}, & c_6=0,\ p\equiv 1\pmod 4\mbox{ and 2-torsion is full},\\
\frac{3(p-1)}{2}, &\mbox{otherwise with full 2-torsion}.
\end{cases}
\end{equation}
The proof of Theorem~\ref{mod2split} will follow immediately from the following
two propositions.
\begin{prop}\label{mod2classcount}
If $p>3$ is prime, then the number of isomorphism classes of elliptic curves 
possessing 2-torsion over $\gf_p$ is given by
\begin{equation*}
N_{2,p}=\begin{cases}
\frac{4p+8}{3}, &\mbox{ if } p\equiv 1\pmod{12},\\
\frac{4p+4}{3}, &\mbox{ if } p\equiv 5\pmod{12},\\
\frac{4p+2}{3}, &\mbox{ if } p\equiv 7\pmod{12},\\
\frac{4p-2}{3}, &\mbox{ if } p\equiv 11\pmod{12}.
\end{cases}
\end{equation*}
\end{prop}
\begin{proof}
In view of (\ref{2torclasscount}), 
we want to count the number of curves $E_{b,c}$ that fall into each of 
six categories.
Let $A_1$ denote the number of curves with cyclic 2-torsion and $c_4=0$, 
$A_2$ denote the number of curves with cyclic 2-torsion and $c_6=0$, $A_3$ 
denote the number of curves with cyclic 2-torsion and $c_4c_6\neq 0$, $A_4$ 
denote the number of curves with full 2-torsion and $c_4=0$, $A_5$ denote 
the number of curves with full 2-torsion and $c_6=0$ and $A_6$ denote the 
number of curves with $c_4c_6\neq 0$. 
Then $N_{2,p}$ can be computed by determining
$A_i$ for $i=1,\dots, 6$ and applying (\ref{2torclasscount}).

Now, an elliptic curve $E_{b,c}$ 
has full 2-torsion if and only if
$b^2-4c$ is a square modulo $p$.
Thus, the number of curves possessing full 2-torsion over $\gf_p$ 
is given by
\begin{equation}\label{full2count}
\sum_{\substack{b=0,\\b^2\neq 4c}}^{p-1}\sum_{c=1}^{p-1}
\frac{1}{2}\left[\legendre{b^2-4c}{p}+1\right]
=\frac{(p-1)(p-2)}{2},
\end{equation}
and the number of curves possessing cyclic 2-torsion over $\gf_p$
is given by
\begin{equation}\label{cyclic2count}
\sum_{\substack{b=0,\\b^2\neq 4c}}^{p-1}\sum_{c=1}^{p-1}
-\frac{1}{2}\left[\legendre{b^2-4c}{p}-1\right]
=\frac{p(p-1)}{2}.
\end{equation}

Note that if $c_4=0$, then 
$b^2\equiv 3c\pmod p$ and hence $\legendre{c}{p}=\legendre{3}{p}$.
Thus, there are $p-1$ nonsingular 
curves (\ref{2torcurve}) that give $c_4=0$.
If a nonsingular curve $E_{b,c}$ possesses full 
2-torsion and $c_4=0$, then
$1=\legendre{b^2-4c}{p}=\legendre{-c}{p}=\legendre{-3}{p}=\legendre{p}{3}$.
Thus, when $p\equiv 1\pmod{3}$, all $p-1$ nonsingular curves $E_{b,c}$ 
with $c_4=0$ will have full 2-torsion, and when $p\equiv 2\pmod 3$, all will 
have cyclic 2-torsion.
Thus,
\begin{eqnarray*}
 A_1&=&\begin{cases}
      0, & p\equiv 1\pmod 3,\\
      p-1, & p\equiv 2\pmod 3,
     \end{cases}\\
 A_4&=&\begin{cases}
      p-1, & p\equiv 1\pmod 3,\\
      0, & p\equiv 2\pmod 3.
     \end{cases}
\end{eqnarray*}

Similar computations lead to
\begin{eqnarray*}
 A_2&=&\frac{p-1}{2},\\
 A_5&=&\frac{3(p-1)}{2}.
\end{eqnarray*}
Finally, using 
(\ref{full2count}) and (\ref{cyclic2count}), we see that
\begin{eqnarray*}
 A_3&=&\begin{cases}
        \frac{(p-1)^2}{2}, & p\equiv 1\pmod 3,\\
        \frac{(p-3)(p-1)}{2}, & p\equiv 2\pmod 3,
       \end{cases}\\
 A_6&=&\begin{cases}
        \frac{(p-1)(p-7)}{2}, & p\equiv 1\pmod 3,\\
        \frac{(p-1)(p-5)}{2}, & p\equiv 2\pmod 3.
       \end{cases}
\end{eqnarray*}

Combining these with (\ref{2torclasscount}), the result follows.
\end{proof}
We now compute the correction term in (\ref{Nmp}).
\begin{prop}\label{mod2corrterm}
The value of the correction term is given by
\begin{equation*}
\sum_{\substack{|r|<2\sqrt p,\\r\equiv 0\pmod 2}}c_{r,p}
=\begin{cases}
10/3, & p\equiv 1\pmod{12},\\
2, & p\equiv 5\pmod{12},\\
4/3, & p\equiv 7\pmod{12},\\
0, & p\equiv 11\pmod{12}.
\end{cases}
\end{equation*}
\end{prop}
\begin{proof}
By (\ref{crp}), we see that each form proportional to $x^2 + xy + y^2$ 
contributes 2/3 to the sum while each form proportional to $x^2 + y^2$
contributes 1/2.

Forms proportional to $x^2 + xy + y^2$ arise for
those
$r \equiv 0 \pmod{2}$ for which there exists $\alpha\in\zz\backslash\{0\}$
such that $r^2-4p=-3\alpha^2$.  
Thus, $p=\left( \frac{r +\alpha i\sqrt{3}}{2}\right)
\left( \frac{r -\alpha i\sqrt{3}}{2}\right)$.
Recall that $p$ factors
in $\zz\left[\frac{1+i\sqrt{3}}{2}\right]$
if and only if $p\equiv 1\pmod{3}$.  For
each such $p$, there are 6 solutions to the above, but only 
2 with $r$ even.
Thus, for $p\equiv 1\pmod{3}$, we must add
$4/3$ to the correction term, and for $p\equiv 2\pmod{3}$, we add 0 to
the correction term.

Forms proportional to $x^2 + y^2$ arise for those
$r \equiv 0 \pmod{2}$ for which there exists $\alpha\in\zz\backslash\{0\}$
such that $r^2-4p=-4\alpha^2$. 
Thus,
$p=\frac{r^2+4\alpha^2}{4}
=\left(\frac{r}{2}+\alpha i\right)\left(\frac{r}{2}-\alpha i\right)$.
Recall that $p$ factors in $\zz[i]$ if and only if $p\equiv 1\pmod{4}$.
Given a prime $p\equiv 1\pmod 4$, there are 4 choices for $r/2$ and hence 4 
choices for $r$.
So, we have 4 forms and need to add $2$ to the
correction term.  When $p\equiv 3\pmod{4}$ we add 0 to the
correction term.
\end{proof}

Combining the results in Propositions~\ref{mod2classcount} 
and~\ref{mod2corrterm}, we have
\begin{equation*}
\sum_{\substack{|r|<2\sqrt p\\r\equiv 0\pmod 2}}H(4p-r^2)
=N_{2,p}-\sum_{\substack{|r|<2\sqrt p\\r\equiv 0\pmod 2}}c_{r,p}
=\frac{4p-2}{3}.
\end{equation*}
Theorem~\ref{mod2split} now follows from (\ref{fullsum}).
\hfill $\Box$

We now give a sketch of the proof of Theorem~\ref{mod4split}.  The proof 
uses some of computations from the proof of Theorem~\ref{mod2split}.

\vspace{12pt}
\noindent\textit{Proof Sketch of Theorem~\ref{mod4split}}.

For $c\equiv\pm 1\pmod 4$, the identities 
\begin{equation*}
\sum_{\substack{|r|<2\sqrt p\\r\equiv c\pmod 2}}H(4p-r^2)
=\frac{p+1}{3}
\end{equation*}
follow directly from Theorem~\ref{mod2split} and (\ref{symmetry}).

By (\ref{numb_classes}), 
$$\sum_{\substack{|r|<2\sqrt p\\r\equiv p+1\pmod 4}}
\left(H(4p-r^2)+c_{r,p}\right)$$
is equal to the number of isomorphism classes of elliptic curves over $E/\gf_p$ 
with $4|\#E(\gf_p)$.  This is equal to the number of classes of curves 
having full 2-torsion plus the number of classes having 
cyclic 4-torsion over $\gf_p$.

As with the Proof of Theorem~\ref{mod2split}, the identities may be checked 
directly for $p=3$.  So, we will assume that $p>3$.
From the proof of Proposition~\ref{mod2classcount}, we see that the number 
of isomorphism classes of curves having full 2-torsion over $\gf_p$ is given by
\begin{equation*}
\begin{cases}
\frac{p+5}{3},& p\equiv 1\pmod{12},\\
\frac{p+1}{3},& p\equiv 5\pmod{12},\\
\frac{p+2}{3},& p\equiv 7\pmod{12},\\
\frac{p-2}{3},& p\equiv 11\pmod{12}.
\end{cases}
\end{equation*}

Following~\cite[pp. 145-147]{knapp:1992}, we see that given any curve 
with 4-torsion over $\gf_p$, we can move the point of order 4 to the origin 
and place the resulting curve into Tate normal to find a model for the curve of the form
\begin{equation}\label{4torcurve}
E_b: y^2+xy-by=x^3-bx^2,
\end{equation}
which has discriminant $\Delta_b=b^4(1+16b)$.  Let $P=(0,0)$ denote the point 
of order 4 on $E_b$.
Thus, as $b$ runs over all of $\gf_p$, we see every class of elliptic curve 
possessing 4-torsion 
over $\gf_p$. As before, we will omit $b=0, 16^{-1}$ from consideration since 
these lead to singular curves.

Given a curve of the form (\ref{4torcurve}), we note that both 
$P=(0,0)$ and $-P$ have order 4.  We see that moving $-P$ to origin and placing 
the resulting curve in Tate normal form gives us exactly the same normal form
as before.  Thus, there is exactly 
one way to represent each cyclic 4-torsion curve in the form (\ref{4torcurve}). 

We are only interested in counting the classes which have cyclic 4-torsion and 
not full 2-torsion (since these have already been counted above).  Thus, 
given a curve (\ref{4torcurve}), we move $2P$ to the origin and place the 
resulting curve in the form (\ref{2torcurve}).  Thus, we see that the curve 
has full 2-torsion if and only if $\legendre{16b+1}{p}=1$.  Hence, we conclude 
that there are $(p-1)/2$ isomorphism classes of curves possessing 
cyclic 4-torsion but not possessing full 2-torsion 
over $\gf_p$.

Finally, in a manner similar to the proof of Proposition~\ref{mod2corrterm},
we check that
\begin{equation*}
\sum_{\substack{|r|<2\sqrt p,\\ r\equiv p+1\pmod 4}}c_{r,p}
=\begin{cases}
7/3,& p\equiv 1\pmod{12},\\
1,& p\equiv 5\pmod{12},\\
4/3,& p\equiv 7\pmod{12},\\
0,& p\equiv 11\pmod{12}.
\end{cases}
\end{equation*}
Combining all the pieces, the result follows.
\hfill $\Box$

For the remainder of this section, we will need the 
following result, which allows us to avoid the problem of detecting full 
$m$-torsion by requiring that our primes satisfy $p\not\equiv 1\pmod m$.

\begin{prop}
If $E$ is an elliptic curve possessing full $m$-torsion over $\gf_p$,
then $p\equiv 1\pmod m$.
\end{prop}
\begin{proof}
Let $G$ be the Galois group of $\gf_p(E[m])/\gf_p$.  
Then $G=\langle\phi\rangle$, where $\phi: \gf_p(E[m])\rightarrow\gf_p(E[m])$ is 
the Frobenius automorphism.
We have the representation 
\begin{equation*}
\rho_m:G\onetoone\Aut(E[m])
\isom\Aut\left(\zz/m\zz\times\zz/m\zz\right)\isom\GL_2(\zz/m\zz).
\end{equation*}
See~\cite[p. 89-90]{sil:1986}.

Now, suppose that $E$ has full $m$-torsion.  Then $\gf_p(E[m])/\gf_p$ is a 
trivial extension.  Whence, $G$ is trivial and  
$\rho_m(\phi)=I\in\GL_2(\zz/m\zz)$.  Therefore, 
applying~\cite[Prop. V.2.3.]{sil:1986}, we have
$p\equiv\det(\rho_m(\phi))\equiv 1\pmod m$.
\end{proof}

We omit the proof of Theorem~\ref{partialmod5} since it is similar to, but 
less involved than the following cases of Theorem~\ref{partialmod7}.

\vspace{12pt}
\noindent\textit{Proof Sketch of Theorem~\ref{partialmod7}
(Cases: $p\not\equiv 0,1\pmod 7; c\equiv \pm (p+1)\pmod 7$)}.

If $p=3$, the identities may be checked directly.  We will assume that 
$p\neq 3,7$ and prime.  Since we also assume that $p\not\equiv 1\pmod 7$, 
we know that no curve may have full 7-torsion over $\gf_p$.  Thus, if $P$ is a  
point of order 7, $E[7](\gf_p)=\langle P\rangle\isom\zz/7\zz$.  

Now, suppose that $E$ possesses 7-torsion, and let $P$ be a point of order 7. 
In a manner similar to~\cite[pp. 145-147]{knapp:1992}, we see that we can move 
$P$ to the origin and put the resulting equation into Tate normal form to 
obtain a model for $E$ of the form
\begin{equation}\label{7torcurve}
E_s: y^2+(1-s^2+s)xy-(s^3-s^2)y=x^3-(s^3-s^2)x^2,
\end{equation}
which has discriminant $\Delta_s=s^7(s-1)^7(s^3-8s^2+5s+1)$.

First, we examine the discriminant.
We note that $s=0,1$ both result in singular curves and so we omit 
these values from consideration.  The cubic $s^3-8s^2+5s+1$ has discriminant 
$7^4$ and hence has Galois group isomorphic to $\zz/3\zz$
(See~\cite[Cor. V.4.7]{hung:1974}).  
Thus, the 
splitting field for the cubic is a degree 3 extension over $\qq$; and 
we see that the cubic will either be irreducible or split completely over 
$\gf_p$.  One can then check that the cubic splits over the cyclotomic field 
$\qq(\zeta_7)$, where $\zeta_7$ is a primitive 7th root of unity.
$\qq(\zeta_7)$ has a unique subfield which is cubic over $\qq$, namely
$\qq(\zeta_7+\zeta_7^6)$.  Thus, $\qq(\zeta_7+\zeta_7^6)$ is the splitting 
field for the cubic $s^3-8s^2+5s+1$.  By examining the way that 
rational primes split in $\qq(\zeta_7)$, one can deduce that rational primes 
are inert in $\qq(\zeta_7+\zeta_7^6)$ unless $p\equiv\pm 1\pmod 7$, in which 
case they split completely.  Thus, we see that the cubic $s^3-8s^2+5s+1$ has 
exactly 3 roots over $\gf_p$ if $p\equiv\pm 1\pmod 7$ and is irreducible 
otherwise.  Hence, as $s$ ranges over all of $\gf_p$, we see $p-5$ 
nonsingular curves (\ref{7torcurve}) if $p\equiv\pm 1\pmod 7$ and $p-2$ 
nonsingular curves (\ref{7torcurve}) otherwise.

Second, we check that the mapping $s\mapsto (1-s^2+s, -(s^3-s^2))$ is a one 
to one mapping of $\gf_p\backslash\{0,1\}$ into $\gf_p^2$.  
Hence, as $s$ ranges over all of 
$\gf_p\backslash\{0,1\}$,
we see $p-2$ distinct equations of the form (\ref{7torcurve}).

Next, we check 
that if we choose to move $-P$ to the origin instead of $P$, we will obtain 
exactly the same Tate normal form for $E$.  Moving $2P$ or $3P$ to the origin 
each result in different normal forms unless $s=0,1$ or is a 
nontrivial cube root of $-1$, in which case both give exactly the same normal 
form as moving $P$ to the origin.  Note that by the above argument, moving 
$-2P$ to the origin will give the same normal form as $2P$ and moving 
$-3P$ to the origin will give the same normal form as $3P$. Now, $s=0,1$ both 
give singular curves; and nontrivial cube roots of $-1$ exists in $\gf_p$ 
if and only if 
$p\equiv 1\pmod 3$, in which case there are exactly 2.  Thus, the number of 
isomorphism classes of curves possessing 7-torsion over $\gf_p$ is given by
\begin{equation*}
\begin{cases}
\frac{p+2}{3},& p\not\equiv\pm 1\pmod 7,\ p\equiv 1\pmod 3,\\
\frac{p-1}{3},& p\equiv 6\pmod 7,\ p\equiv 1\pmod 3,\\
\frac{p-2}{3},& p\not\equiv\pm 1\pmod 7,\ p\equiv 2\pmod 3,\\
\frac{p-5}{3},& p\equiv 6\pmod 7,\ p\equiv 2\pmod 3.
\end{cases}
\end{equation*}

Finally, we check that, for $p\not\equiv 1\pmod 7$,
\begin{equation*}
\sum_{\substack{|r|<2\sqrt p,\\ r\equiv p+1\pmod 7}}c_{r,p}
=\begin{cases}
4/3,& p\equiv 1\pmod 3,\\
0,&\mbox{otherwise.}
\end{cases}
\end{equation*}
The result now follows for $c\equiv\pm (p+1)\pmod 7$ by 
(\ref{numb_classes}) and (\ref{symmetry}).
\hfill $\Box$

The remaining cases of Theorem~\ref{partialmod7} will be treated in 
Section~\ref{trace_method}.

\section{Modular Forms and Hurwitz Class Numbers}\label{mf_method}

We do not give an exhaustive account of modular forms.  Instead 
we refer the reader to Miyake's book~\cite{miy:1989} and Shimura's 
paper~\cite{shi:1973} for the details.  Recall the definition of a modular form.
\begin{defn}\label{mf}
Let $f:\mathfrak{h}\rightarrow\cc$ be holomorphic, let $k\in\frac{1}{2}\zz$, and let 
$\chi$ be a Dirichlet character modulo $N$.  
Then $f$ is said to be a 
modular form of 
weight $k$, level $N$, and character $\chi$ if
\begin{enumerate}
\item 
$\displaystyle f(\gamma z)=\begin{cases}
\chi(d)(cz+d)^kf(z), &\mbox{ if }k\in\zz,\\
\chi(d)\legendre{c}{d}\legendre{-4}{d}^{-k}(cz+d)^kf(z), 
&\mbox{ if }k\in 1/2+\zz
\end{cases}$\newline
for all $\gamma=\abcd\in\Gamma_0(N)$;
\item $f$ is 
holomorphic at the cusps of $\mathfrak{h}/\Gamma_0(N)$.
\end{enumerate}
\end{defn}
We denote this space by $\mathscr{M}_k(N,\chi)$.  In the case 
that $\chi$ is trivial, we will omit the character.  We also recall that 
the space is a finite dimensional vector space over $\cc$ and decomposes as 
\begin{equation*}
\mathscr{M}_k(N,\chi)=\mathscr{E}_k(N,\chi)\oplus\mathscr{S}_k(N,\chi),
\end{equation*}
where $\mathscr{S}_k(N,\chi)$ is the subspace of cusp forms and 
$\mathscr{E}_k(N,\chi)$ is the Eisenstein subspace.

For the remainder, we put $q:=q(z)=e^{2\pi iz}$.
It is well-known that modular forms have natural representations 
as Fourier series.  Here we will develop some notation and show 
how to construct weight 2 Eisenstein series.  Given two Dirichlet 
characters $\psi_1,\psi_2$ with conductors $M_1$ and $M_2$ 
respectively, put $M=M_1M_2$; and put
\begin{equation*}
E_2(z;\psi_1,\psi_2):=\sum_{n=0}^\infty a_nq^n,
\end{equation*}
where
\begin{equation*}
a_0=\begin{cases}
0, &\mbox{if }\psi_1\mbox{ is non-trivial},\\ 
\frac{M-1}{24}, &\mbox{if $\psi_1$ and $\psi_2$ are both trivial},\\
-B_{k, (\psi_1\psi_2)}/4, &\mbox{otherwise,}
\end{cases}
\end{equation*}
and, for $n\ge 1$,
\begin{equation*}
a_n=\sum_{d|n}\psi_1(n/d)\psi_2(d)d.
\end{equation*}
Here $B_{k,\chi}$ denotes the $k^{\mathrm{th}}$ generalized Bernoulli number 
associated to $\chi$, whose generating function is 
$F_\chi(t)=\sum_{a=1}^m\frac{\chi(a)te^{at}}{e^{mt}-1},$ where $m$ is the 
conductor of $\chi$.
Then, subject to a couple of technical conditions on $\psi_1$ and $\psi_2$,
one can show
$E_2(z;\psi_1,\psi_2)\in\mathscr{E}_2(M,\psi_1\psi_2)$.  
See~\cite[pp. 176-181]{miy:1989}.

When computing with modular forms, the following result 
allows us to work with only a finite number of coefficients.
\begin{thm}[See Prop. 1.1 in \cite{fre:1994}]\label{coefftbound}
Suppose that $k\in\zz$ and $f(z)=\sum a_nq^n\in\mathscr{M}_k(N,\chi)$.  
Put $m=\frac{k}{12}N\prod_{p|N}\left(1+\frac{1}{p}\right)$.  Then
$f(z)$ is uniquely 
determined by its Fourier coefficients $a_0,a_1,\dots, a_{[m]}$.
\end{thm}

In~\cite[p. 90]{zag:1976} (\cite{zag:1975} also), Zagier defines a 
non-holomorphic $q$-series whose holomorphic part is the generating
function for the Hurwitz class number, $H(N)$.  He then shows that the series 
transforms like a modular form of weight 3/2 on $\Gamma_0(4)$.
In~\cite[Cor 3.2]{coh:1975}, Cohen points out that a holomorphic form 
can be obtained by only summing over those $N$ that fall into certain 
arithmetic progressions.  He states without proof that the resulting 
series should be a form on $\Gamma_0(A)$, where he specifies $A$.
However, using techniques similar to those found 
in~\cite[pp. 128-129]{kob:1993}, we 
were only able to prove the following version of this result.
\begin{thm}\label{mfandH}
If $-b$ is a quadratic non-residue modulo $a$, then
\begin{equation*}
 \mathscr{H}_1(z;a,b):=\sum_{N\equiv b\pmod{a}}H(N)q^N
 \in\mathscr{M}_{3/2}(G_a),
\end{equation*}
where 
\begin{equation*}
G_a=\left\{\left(\begin{array}{cc}
\alpha &\beta\\
\gamma &\delta
\end{array}\right)\in\Gamma_0(A):\alpha^2\equiv 1\pmod a\right\},
\end{equation*}
and we take $A=a^2$ if $4|a$ and $A=4a^2$ otherwise.
\end{thm}
We now turn to the proof of Theorem~\ref{mod3split}.

\vspace{12pt}
\noindent\textit{Proof of Theorem~\ref{mod3split}}.

It is well-known that the classical theta series
$\theta(z):=\sum_{s=-\infty}^\infty q^{s^2}\in\mathscr{M}_{1/2}(4)$.
Applying Theorem~\ref{mfandH}, we see that
$\displaystyle\mathscr{H}_1(z;3,1)
=\sum_{N\equiv 1\pmod{3}}H(N)q^N\in\mathscr{M}_{3/2}(36)$.
Note that in this case, $G_3=\Gamma_0(36)$.
Thus, we can check that product
$\mathscr{H}_1(z;3,1)\theta(z)\in\mathscr{M}_2(36)$.  Observe that
the coefficients of the product bear a striking resemblance to the 
sums of interest.  Indeed,
\begin{eqnarray*}
\mathscr{H}_1(z;3,1)\theta(z)
&=&\sum_{s=-\infty}^\infty\sum_{N\equiv 1\pmod 3}H(N)q^{N+s^2}\\
&=&\sum_{n\equiv 1\pmod 3}
\left(\sum_{\stackrel{|s|<\sqrt n,}{s\equiv 0\pmod 3}}H(n-s^2)\right)q^n\\
& &+\sum_{n\equiv 2\pmod 3}
\left(\sum_{\stackrel{|s|<\sqrt n,}{s\equiv\pm 1\pmod 3}}H(n-s^2)\right)q^n.
\end{eqnarray*}

We will prove Theorem~\ref{mod3split} by expressing 
$\mathscr{H}_1(z;3,1)\theta(z)$ as a linear combination of 
basis forms with ``nice" Fourier coefficients.  Note that 
Theorem~\ref{coefftbound} says that we will only need to consider 
the first 13 coefficients in order to do this.

Let $\chi_0$ denote the principal character of conductor 1, and 
let $\chi_{0,2}$ and $\chi_{0,3}$ denote the trivial characters 
modulo 2 and 3 respectively.  Finally let $\legendre{\cdot}{3}$ denote the 
Legendre symbol modulo 3.  Then one can show that $\mathscr{E}_2(36)$ 
has dimension 11 over $\cc$ and is spanned by
$$\left\{
\begin{array}{ccc}
E_2(z;\chi_0,\chi_{0,2}), &
E_2(z;\chi_0,\chi_{0,3}), &
E_2(z;\legendre{\cdot}{3},\legendre{\cdot}{3}),\\
E_2(2z;\chi_0,\chi_{0,2}), &
E_2(3z;\chi_0,\chi_{0,2}), &
E_2(9z;\chi_0,\chi_{0,2}), \\
E_2(6z;\chi_0,\chi_{0,2}), &
E_2(18z;\chi_0,\chi_{0,2}), &
E_2(3z;\chi_0,\chi_{0,3}),\\
E_2(2z;\legendre{\cdot}{3},\legendre{\cdot}{3}),&
E_2(4z;\legendre{\cdot}{3},\legendre{\cdot}{3}) &
\end{array}\right\}.$$
The cusp space $\mathscr{S}_2(36)$ is 1 dimensional and is spanned by the 
cusp form associated to the elliptic curve
\begin{equation*}
E: y^2=x^3+1,
\end{equation*}
which is the inverse Mellin transform of the $L$-series
\begin{eqnarray*}
L(E,s)&=&\prod_{p\dnd 36}(1-a(p)p^{-s}+p^{1-2s})^{-1}
=\sum_{(n,6)=1}\frac{a(n)}{n^s}\\
&=&\sum_{(n,6)=1}
\prod_{p|n}\left[
\sum_{\left\lceil\frac{\ord_p(n)}{2}\right\rceil\le k\le \ord_p(n)}
\binom{k}{\ord_p(n)-k}a(p)^{2k-\ord_p(n)}(-p)^{\ord_p(n)-k}\right]\frac{1}{n^s},
\end{eqnarray*}
where $a(p):=p+1-\#E(\gf_p)$, and we take $a(p)^0=1$ even if $a(p)=0$.  
We will denote this cusp form by $f_E(z)$.

One can verify computationally that
\begin{eqnarray*}
\mathscr{H}_1(z;3,1)\theta(z)
&=& 
\frac{-1}{16}E_2(z;\chi_0,\chi_{0,2})
+\frac{3}{16}E_2(z;\chi_0,\chi_{0,3})
+\frac{-1}{24}E_2(z;\legendre{\cdot}{3},\legendre{\cdot}{3})\\
&+&\frac{-1}{2}E_2(2z;\chi_0,\chi_{0,2})
+\frac{1}{4}E_2(3z;\chi_0,\chi_{0,2})
+\frac{-3}{16}E_2(9z;\chi_0,\chi_{0,2})\\
&+&2E_2(6z;\chi_0,\chi_{0,2})
+\frac{-3}{2}E_2(18z;\chi_0,\chi_{0,2})
+\frac{-3}{16}E_2(3z;\chi_0,\chi_{0,3})\\
&+&\frac{-1}{8}E_2(2z;\legendre{\cdot}{3},\legendre{\cdot}{3})
+\frac{-1}{3}E_2(4z;\legendre{\cdot}{3},\legendre{\cdot}{3})
+\frac{-1}{12}f_E(z).
\end{eqnarray*}

Let $\sigma(n):=\sigma_1(n)=\sum_{d|n}d$,
and define the arithmetic functions
\begin{eqnarray*}
 \mu_1(n)&:=&\sum_{\substack{d|n,\\ d\not\equiv 0\pmod 2}}d,\\
 \mu_2(n)&:=&\sum_{\substack{d|n,\\ d\not\equiv 0\pmod 3}}d,\\
 \mu_3(n)&:=&\legendre{n}{3}\sigma(n).
\end{eqnarray*}
Extend these to $\qq$ by setting
$\mu_i(r)=0$ for $r\in\qq\backslash\zz$ ($i=1,2,3$).
Comparing $n$-th coefficients, we have the following proposition.
\begin{prop}\label{techprop}
\begin{eqnarray*}
\sum_{|s|<\sqrt n}\substack{*\\ \\ \ }H(n-s^2)
=&-&\frac{1}{16}\mu_1(n)
+\frac{3}{16}\mu_2(n)
-\frac{1}{24}\mu_3(n)
-\frac{1}{2}\mu_1(n/2)\\
&+&\frac{1}{4}\mu_1(n/3)
-\frac{3}{16}\mu_1(n/9)
+2\mu_1(n/6)
-\frac{3}{2}\mu_1(n/18)\\
&-&\frac{3}{16}\mu_2(n/3)
-\frac{1}{8}\mu_3(n/2)
-\frac{1}{3}\mu_3(n/4)
-\frac{1}{12}a(n),
\end{eqnarray*}
where the * denotes the fact that 
if $n\equiv 1\pmod 3$, we take the sum over all 
$s\equiv 0\pmod 3$; if $n\equiv 2\pmod 3$, we take 
the sum over all $s\equiv\pm 1\pmod 3$.
\end{prop}

Using (\ref{symmetry}) and the fact that $a(n)=0$ if $(n,6)>1$, 
we are able to extract the identities
\begin{equation*}
\sum_{\substack{|r|<2\sqrt p,\\ r\equiv c\pmod 3}}H(4p-r^2)
=\begin{cases}
\frac{p+1}{2}, &\mathrm{ if }\ p\equiv 1\pmod 3,\ c=0,\\
\frac{p+1}{2}, &\mathrm{ if }\ p\equiv 2\pmod 3,\ c= 1\mbox{ or } 2.
\end{cases}
\end{equation*}
So, using (\ref{fullsum}), we are able to obtain Theorem~\ref{mod3split}
as a corollary.

At this point, we also note that there is a more general identity than 
(\ref{fullsum}), which is due to Hurwitz.  See~\cite[p. 236]{coh:1993}. 
In particular, if we let $\lambda(n):=\frac{1}{2}\sum_{d|n}\min(d,n/d)$, then
\begin{equation*}
\sum_{|r|<2\sqrt N}H(4N-r^2)=2\sigma(N)-2\lambda(N).
\end{equation*}
So, this identity together with Proposition~\ref{techprop} makes it possible to 
generalize Theorem~\ref{mod3split} for $p$ not necessarily prime.

In addition, if we study the cusp form in our basis carefully, we can extract
other nice formulae as well.  For example, we can use the fact 
that $a(p)=0$ if $p\equiv 2\pmod 3$ to obtain Theorem~\ref{anothresult}.

\section{The Eichler-Selberg Trace Formula and Hurwitz Class Numbers}
\label{trace_method}

The Eichler-Selberg trace formula gives the value of the
trace of the $n$-th Hecke operator acting on $\mathscr{S}_k(N,\psi)$.
The following gives a formula for computing the 
trace in a specific setting, which is useful for our purposes.  
The trace formula is, in fact, 
much more general and 
we refer the reader to~\cite[pp. 12-13]{hps:1989} for 
more details.
\begin{thm}Let $p$ be a prime.
The trace of the $p$-th Hecke operator acting on 
$\mathscr{S}_2(N)$ is given by
\begin{eqnarray*}
\mbox{tr}_{2,N}(T_p)=p+1
&-&\sum_{s\in H}\frac{1}{p-1}\sum_{f|t}\frac{1}{2}\phi((s^2-4p)^{1/2}/f)
\prod_{l|N}c(s,f,l)\\
&-&\sum_{s\in E_1\cup E_2}\frac{1}{2}
\sum_{f|t}\frac{h((s^2-4p)/f^2)}{\omega((s^2-4p)/f^2)}\prod_{l|N}c(s,f,l),
\end{eqnarray*}
where $H=\{s:s^2-4p=t^2\}$,
$E_1=\{s:s^2-4p=t^2m,\ 0>m\equiv 1\pmod 4\}$, and
$E_2=\{s:s^2-4p=t^24m,\ 0>m\equiv 2,3\pmod 4\}$.
\end{thm}
Here, $\phi$ is the Euler $\phi$-function. For $d<0$,
$h(d)$ is the class number of the order of $\qq(\sqrt d)$ of 
discriminant $d$, and $\omega(d)$ is 1/2 the cardinality of its unit group. 
For a prime $l|N$, 
$c(s,f,l)$ essentially counts the number of solutions to a certain 
system of congruences.

It is well-known that if $s\in E_1\cup E_2$, then
\begin{equation*}
H(4p-s^2)=\sum_{f|t}\frac{h((s^2-4p)/f^2)}{\omega((s^2-4p)/f^2)}.
\end{equation*}
In fact, one may even define the Hurwitz class number in this way.
See~\cite[p. 318]{cox:1989}.
So, if it is possible to control the $c(s,f,l)$ --  in 
particular, if it is possible to make them constant with respect to $f$, 
then there is hope that Hurwitz class number relations may be extracted from 
the trace formula.  Indeed, if $p$ is quadratic nonresidue modulo $l$ for 
all primes $l$ dividing $N$, then the computation of $c(s,f,l)$ is quite simple 
and does not depend on $f$.

%

For example, if we apply the trace formula to $T_p$ acting on 
$\mathscr{S}_2(7)=\{0\}$ for 
$p\equiv 3,5,6\pmod 7$, 
\begin{equation*}
c(s,f,l)=\begin{cases}
2,& p\equiv 3\pmod 7,\ s\equiv 0,\pm 3,\\
2,& p\equiv 5\pmod 7,\ s\equiv 0,\pm 1,\\
2,& p\equiv 6\pmod 7,\ s\equiv 0,\pm 2,\\
0,&\mbox{otherwise.}
\end{cases}
\end{equation*}
The resulting Hurwitz class number relation is the following.

\begin{prop}
\begin{equation*}
\sum_{|s|<2\sqrt p}\substack{* \\ \\ \ }H(4p-s^2)=p-1,
\end{equation*}
where the * denotes the fact that
if $p\equiv 3\pmod 7$, the sum is over all $s\equiv 0,\pm 3\pmod 7$; if 
$p\equiv 5\pmod 7$, the sum is over all $s\equiv 0,\pm 1\pmod 7$; and if 
$p\equiv 6\pmod 7$, the sum is over all $s\equiv 0,\pm 2\pmod 7$. 
\end{prop} 
Combining the above proposition with the cases of Theorem~\ref{partialmod7} 
that were proven in Section~\ref{elliptic_method}, we are able to obtain the 
remaining formulae in the theorem.

\section{Conjectures}\label{conj}

For all primes $p$ sufficiently large, 
the tables in this section give conjectured 
values for the sum
$$\sum_{\substack{|r|<2\sqrt p, \\ r \equiv c \pmod{m}}}H(4p-r^2).$$
These values have been checked for primes $p<1,000,000$.  Where an entry is 
bold and 
marked by asterisks, the 
formula is handled by a theorem somewhere in this paper; where an entry is 
blank, we were not able to recognize any simple pattern from the computations.

We note that for $m=5$ and $7$, neither the curve counting method of 
Section~\ref{elliptic_method} alone nor the basis of modular forms approach of 
Section~\ref{mf_method} alone will be sufficient for a complete 
characterization of these sums.  Rather a combination of the two methods should 
work.

\begin{table}[h]
\begin{center}
\begin{tabular}{|c|c|c|c|}
\hline
&$c=0$&$c=\pm 1$&$c=\pm 2$\\
\hline
$p\equiv 1\pmod 5$&$\frac{(p+1)}{2}$&$\frac{(p+1)}{3}$&$\frac{(5p-7)}{12}$\\
\hline
$p\equiv 2\pmod 5$&$\frac{(p+1)}{3}$&$\frac{(p+1)}{3}$&\boldmath $*\frac{(p-1)}{2}*$\\
\hline
$p\equiv 3\pmod 5$&$\frac{(p+1)}{3}$&\boldmath $*\frac{(p-1)}{2}*$&$\frac{(p+1)}{3}$\\
\hline
$p\equiv 4\pmod 5$&\boldmath $*\frac{(p-3)}{2}*$&$\frac{(5p+5)}{12}$&$\frac{(p+1)}{3}$\\
\hline
\end{tabular}
\end{center}
\caption{$m=5$}
\end{table}

\begin{table}[h]
\begin{center}
\begin{tabular}{|c|c|c|c|c|}
\hline
&$c=0$&$c=\pm 1$&$c=\pm 2$&$c=\pm 3$\\
\hline
$p\equiv 1\pmod 7$&&$\frac{(p+1)}{3}$&&\\
\hline
$p\equiv 2\pmod 7$&&&&\boldmath$*\frac{(p-2)}{3}*$\\
\hline
$p\equiv 3\pmod 7$&\boldmath$*\frac{(p+1)}{3}*$&$\frac{(p+1)}{4}$&$\frac{(p+1)}{4}$&\boldmath$*\frac{(p-2)}{3}*$\\
\hline
$p\equiv 4\pmod 7$&&&\boldmath$*\frac{(p-2)}{3}*$&\\
\hline
$p\equiv 5\pmod 7$&\boldmath$*\frac{(p+1)}{3}*$&\boldmath$*\frac{(p-2)}{3}*$&$\frac{(p+1)}{4}$&$\frac{(p+1)}{4}$\\
\hline
$p\equiv 6\pmod 7$&\boldmath$*\frac{(p-5)}{3}*$&$\frac{(p+1)}{4}$&\boldmath$*\frac{(p+1)}{3}*$&$\frac{(p+1)}{4}$\\
\hline
\end{tabular}
\end{center}
\caption{$m=7$}
\end{table}

For making further progress on the above tables,
it appears that the method of Section~\ref{mf_method} will be most fruitful.  
The difficulty in using this method is that, in each
case, the group $G_m$ (defined in Theorem~\ref{mfandH})
is strictly contained in $\Gamma_0(4m^2)$. 
So, we will need a much larger basis of modular forms. Certainly a basis for 
$\mathscr{M}_2(\Gamma_1(4m^2))$ would be sufficient.
However, to completely fill in the table, another method will be needed.
Perhaps an adaptation of the curve counting method of 
Section~\ref{elliptic_method} for $p\equiv 1\pmod m$ would be best.  
The main obstacle to overcome is that one must deal with 
the presence of 
full $m$-torsion curves when $p\equiv 1\pmod m$.  

In general, for a prime $m$, we
note that the curve counting method will give proofs for 
$c\equiv\pm (p+1)\pmod m$.
We also note that, for primes $m\equiv 1\pmod 4$, the basis of forms method will 
give proofs for the cases when $4p-c^2$ is not a square; and for primes 
$m\equiv 3\pmod 4$, the basis of forms method will give proofs for the cases 
when $4p-c^2$ is a square.  Thus, for a prime $m$, we will obtain 
half the cases from the basis of forms approach and we will obtain 
$2(m-1)+1$ more from the curve counting approach assuming that the case 
$p\equiv 1\pmod m$ can be adequately handled.  So, at least for primes
greater than 7, a third method will be necessary to fully characterize 
how the sum splits.

\section{Acknowledgments}

The authors would like to thank Ken Ono for suggesting that we consider the 
Eichler-Selberg trace formula.  The authors would also like to thank Robert 
Osburn for helpful conversation.

\bibliography{references}

\begin{thebibliography}{10}

\bibitem{coh:1975}
Henri Cohen.
\newblock Sums involving the values at negative integers of {$L$}-functions of
  quadratic characters.
\newblock {\em Math. Ann.}, 217(3):271--285, 1975.

\bibitem{coh:1993}
Henri Cohen.
\newblock {\em A Course in Computational Algebraic Number Theory}.
\newblock Springer-Verlag, New York, 1996.

\bibitem{cox:1989}
David~A. Cox.
\newblock {\em Primes of the Form $x^2+ny^2$}.
\newblock Wiley-Interscience, New York, 1989.

\bibitem{deu:1941}
Max Deuring.
\newblock Die {T}ypen der {M}ultiplikatorenringe elliptischer
  {F}unktionenk\"orper.
\newblock {\em Abh. Math. Sem. Hansischen Univ.}, 14:197--272, 1941.

\bibitem{eic:1956}
Martin Eichler.
\newblock On the class of imaginary quadratic fields and the sums of divisors
  of natural numbers.
\newblock {\em J. Indian Math. Soc. (N.S.)}, 19:153--180 (1956), 1955.

\bibitem{fre:1994}
Gerhard Frey.
\newblock Construction and arithmetical applications of modular forms of low
  weight.
\newblock In {\em Elliptic curves and related topics}, volume~4 of {\em CRM
  Proc. Lecture Notes}, pages 1--21. Amer. Math. Soc., Providence, RI, 1994.

\bibitem{hps:1989}
Hiroaki Hijikata, Arnold~K. Pizer, and Thomas~R. Shemanske.
\newblock The basis problem for modular forms on {$\Gamma\sb 0(N)$}.
\newblock {\em Mem. Amer. Math. Soc.}, 82(418):vi+159, 1989.

\bibitem{zag:1976}
F.~Hirzebruch and D.~Zagier.
\newblock Intersection numbers of curves on {H}ilbert modular surfaces and
  modular forms of {N}ebentypus.
\newblock {\em Invent. Math.}, 36:57--113, 1976.

\bibitem{hung:1974}
Thomas~W. Hungerford.
\newblock {\em Algebra}.
\newblock Springer-Verlag, New York, 1974.

\bibitem{knapp:1992}
Anthony~W. Knapp.
\newblock {\em Elliptic Curves}.
\newblock Princeton University Press, Princeton, 1992.

\bibitem{kob:1993}
Neal Koblitz.
\newblock {\em Introduction to Elliptic Curves and Modular Forms}.
\newblock Springer-Verlag, New York, 1993.

\bibitem{len:1987}
H.~W. Lenstra, Jr.
\newblock Factoring integers with elliptic curves.
\newblock {\em Ann. of Math. (2)}, 126(3):649--673, 1987.

\bibitem{miy:1989}
Toshitsune Miyake.
\newblock {\em Modular Forms}.
\newblock Springer-Verlag, New York, 1989.

\bibitem{shi:1973}
Goro Shimura.
\newblock On modular forms of half integral weight.
\newblock {\em Ann. of Math. (2)}, 97:440--481, 1973.

\bibitem{sil:1986}
Joseph~H. Silverman.
\newblock {\em The Arithmetic of Elliptic Curves}.
\newblock Springer-Verlag, New York, 1986.

\bibitem{zag:1975}
Don Zagier.
\newblock Nombres de classes et formes modulaires de poids {$3/2$}.
\newblock {\em C. R. Acad. Sci. Paris S\'er. A-B}, 281(21):Ai, A883--A886,
  1975.

\end{thebibliography}
\end{document}